\documentclass[12pt]{amsart}

\usepackage{amsmath,amssymb,amsbsy,amsfonts,amsthm,latexsym,
                        amsopn,amstext,amsxtra,euscript,amscd,mathrsfs,color,bm}
                        \usepackage{ulem}

\usepackage{float}
\usepackage[english]{babel}
\usepackage{mathtools}
\usepackage{todonotes}
\usepackage{url}
\usepackage[colorlinks,linkcolor=blue,anchorcolor=blue,citecolor=blue,backref=page]{hyperref}

\usepackage[norefs,nocites]{refcheck}

\newtheorem{theorem}{Theorem}
\newtheorem{lemma}[theorem]{Lemma}

\newtheorem{cor}[theorem]{Corollary}

\newtheorem{rem}[theorem]{Remark}

\numberwithin{equation}{section}
\numberwithin{theorem}{section}
\numberwithin{table}{section}

\numberwithin{figure}{section}

\newfont{\teneufm}{eufm10}
\newfont{\seveneufm}{eufm7}
\newfont{\fiveeufm}{eufm5}
\newfam\eufmfam
                \textfont\eufmfam=\teneufm \scriptfont\eufmfam=\seveneufm
                \scriptscriptfont\eufmfam=\fiveeufm
\def\bbbc{{\mathchoice {\setbox0=\hbox{$\displaystyle\rm C$}\hbox{\hbox
to0pt{\kern0.4\wd0\vrule height0.9\ht0\hss}\box0}}
{\setbox0=\hbox{$\textstyle\rm C$}\hbox{\hbox
to0pt{\kern0.4\wd0\vrule height0.9\ht0\hss}\box0}}
{\setbox0=\hbox{$\scriptstyle\rm C$}\hbox{\hbox
to0pt{\kern0.4\wd0\vrule height0.9\ht0\hss}\box0}}
{\setbox0=\hbox{$\scriptscriptstyle\rm C$}\hbox{\hbox
to0pt{\kern0.4\wd0\vrule height0.9\ht0\hss}\box0}}}}
\def\bbbq{{\mathchoice {\setbox0=\hbox{$\displaystyle\rm
Q$}\hbox{\raise 0.15\ht0\hbox to0pt{\kern0.4\wd0\vrule
height0.8\ht0\hss}\box0}} {\setbox0=\hbox{$\textstyle\rm
Q$}\hbox{\raise 0.15\ht0\hbox to0pt{\kern0.4\wd0\vrule
height0.8\ht0\hss}\box0}} {\setbox0=\hbox{$\scriptstyle\rm
Q$}\hbox{\raise 0.15\ht0\hbox to0pt{\kern0.4\wd0\vrule
height0.7\ht0\hss}\box0}} {\setbox0=\hbox{$\scriptscriptstyle\rm
Q$}\hbox{\raise 0.15\ht0\hbox to0pt{\kern0.4\wd0\vrule
height0.7\ht0\hss}\box0}}}}
\def\bbbt{{\mathchoice {\setbox0=\hbox{$\displaystyle\rm
T$}\hbox{\hbox to0pt{\kern0.3\wd0\vrule height0.9\ht0\hss}\box0}}
{\setbox0=\hbox{$\textstyle\rm T$}\hbox{\hbox
to0pt{\kern0.3\wd0\vrule height0.9\ht0\hss}\box0}}
{\setbox0=\hbox{$\scriptstyle\rm T$}\hbox{\hbox
to0pt{\kern0.3\wd0\vrule height0.9\ht0\hss}\box0}}
{\setbox0=\hbox{$\scriptscriptstyle\rm T$}\hbox{\hbox
to0pt{\kern0.3\wd0\vrule height0.9\ht0\hss}\box0}}}}
\def\bbbs{{\mathchoice
{\setbox0=\hbox{$\displaystyle     \rm S$}\hbox{\raise0.5\ht0\hbox
to0pt{\kern0.35\wd0\vrule height0.45\ht0\hss}\hbox
to0pt{\kern0.55\wd0\vrule height0.5\ht0\hss}\box0}}
{\setbox0=\hbox{$\textstyle        \rm S$}\hbox{\raise0.5\ht0\hbox
to0pt{\kern0.35\wd0\vrule height0.45\ht0\hss}\hbox
to0pt{\kern0.55\wd0\vrule height0.5\ht0\hss}\box0}}
{\setbox0=\hbox{$\scriptstyle      \rm S$}\hbox{\raise0.5\ht0\hbox
to0pt{\kern0.35\wd0\vrule height0.45\ht0\hss}\raise0.05\ht0\hbox
to0pt{\kern0.5\wd0\vrule height0.45\ht0\hss}\box0}}
{\setbox0=\hbox{$\scriptscriptstyle\rm S$}\hbox{\raise0.5\ht0\hbox
to0pt{\kern0.4\wd0\vrule height0.45\ht0\hss}\raise0.05\ht0\hbox
to0pt{\kern0.55\wd0\vrule height0.45\ht0\hss}\box0}}}}
\def\bbbz{{\mathchoice {\hbox{$\sf\textstyle Z\kern-0.4em Z$}}
{\hbox{$\sf\textstyle Z\kern-0.4em Z$}} {\hbox{$\sf\scriptstyle
Z\kern-0.3em Z$}} {\hbox{$\sf\scriptscriptstyle Z\kern-0.2em
Z$}}}}

\def\squareforqed{\hbox{\rlap{$\sqcap$}$\sqcup$}}
\def\qed{\ifmmode\squareforqed\else{\unskip\nobreak\hfil
\penalty50\hskip1em\null\nobreak\hfil\squareforqed
\parfillskip=0pt\finalhyphendemerits=0\endgraf}\fi}

\def\cE{{\mathcal E}}

\def\cG{{\mathcal G}}
\def\cH{{\mathcal H}}

\def\cT{{\mathcal T}}
\def\cU{{\mathcal U}}
\def\cV{{\mathcal V}}

\def\cY{{\mathcal Y}}

\def\HH{{\mathsf H}}

\def\le{\leqslant}
\def\leq{\leqslant}
\def\ge{\geqslant}
\def\leq{\leqslant}

\def\va{{\mathbf{a}}}

\def\vr{{\mathbf{r}}}

\newcommand{\ignore}[1]{}

\def\e{\mathbf{e}}

\def \F{\mathbb{F}}

\def \Z{\mathbb{Z}}

\def \Q{\mathbb{Q}}

\def \Z{\mathbb{Z}}

\def\mand{\qquad\mbox{and}\qquad}

\def\Fp{\F_p}

\def\\{\cr}
\def\({\left(}
\def\){\right)}

\def\vh{\mathbf{h}}

\def\e{{\mathbf{\,e}}}
\def\ep{{\mathbf{\,e}}_p}

\newcommand{\ind}{\operatorname{ind}}
\newcommand{\Res}{\operatorname{Res}}

\begin{document}

\title{On a family of sparse exponential sums}

\author[M.~Z.~Garaev]{Moubariz~Z.~Garaev}
\address{Centro  de Ciencias Matem{\'a}ticas,  Universidad Nacional Aut\'onoma de
M{\'e}\-xico, C.P. 58089, Morelia, Michoac{\'a}n, M{\'e}xico}
\email{garaev@matmor.unam.mx}

\author[Z. Rudnick] {Zeev Rudnick}
\address{School of Mathematical Sciences, Tel-Aviv University, Tel-Aviv 69978, Israel}
\email{rudnick@tauex.tau.ac.il}

\author[I. E. Shparlinski] {Igor E. Shparlinski}
\address{School of Mathematics and Statistics, University of New South Wales, Sydney, NSW 2052, Australia}
\email{igor.shparlinski@unsw.edu.au}

\thanks{Z.R. was supported by the European Research Council (ERC) under the European Union's Horizon 2020 research and innovation programme (grant agreement No.~786758) and by the Israel Science Foundation 
(grant No.~1881/20). 
I.S. was supported by the Australian Research Council, grants DP230100530 and DP230100534.}  

\begin{abstract}  
We investigate exponential sums modulo primes whose phase function is a sparse polynomial, with exponents growing with the prime. In particular, such sums model those which appear in the study of the quantum cat map.
While they are not amenable to treatment by algebro-geometric methods such as Weil's bounds,  Bourgain (2005) gave a nontrivial estimate for these and more general sums. In this work we obtain explicit bounds with reasonable savings over various types of averaging. We also initiate the study of the value distribution of these sums. 
\end{abstract}


\date{\today}
\maketitle

\tableofcontents

\section{Introduction}
\subsection{Motivation and set-up}
Let $p$ be a large prime number and let $\F_p$ denote the finite field of $p$ elements, 
which we assume to be represented by the set $\{0,1, \ldots, p-1\}$. 
For $\va=(a_1, \ldots, a_t)\in \Fp^t$, $\vh = (h_1, \ldots, h_t) \in  \(\F_p^*\)^t$, $h_i\neq 0,\pm 1$,  we define the exponential sum 
$$
S(\va, \vh ;p) := \sum_{x=1}^{p-1}\e_p(a_1h_1^x + \ldots + a_th_t^x), 
$$
where $\ep(z) = \exp(2\pi i z/p)$, involving a linear combination of $t\ge 1 $ exponential functions.

It has been  known since the results of~\cite{KR2001} that these sums model those that occur in the study of the quantum cat map~\cite{Bourgain GAFA, KORS}, and it is important to know bounds for them as well as finer results about their value distribution. However, they typically do not fall into the class of  exponential sums for which algebro-geometric methods apply.  
 
Writing
\begin{equation}
\label{eq:prime root repr}
h_i = g^{r_i}, \qquad i =1, \ldots, t,
\end{equation}
for a fixed primitive root $g$ of $\F_p^*$ and $1\leq r_i< p-1$,  studying such sums  becomes equivalent to studying sums with sparse polynomials, that is, sums of the form 
\begin{equation}
\label{eq:sparse sums}
T(\va,\vr;p):= \sum_{x\in \F_p^*}\e_p\(a_1 x^{r_1} + \ldots + a_t x^{r_t}\) .
\end{equation}

 Our goal in this note is to explore bounds for these sums and to initiate the study of their value distribution.
 
\subsection{Upper bounds} 

Weil's bound~\cite[Appendix~V, Lemma~5]{Weil}  gives that   
\begin{equation}
\label{eq:Weil}
|T(\va,\vr;p)|\leq (\max_{i=1,\ldots ,t} r_i)\sqrt{p}.
\end{equation} 
This is optimal for bounded values of $r_i$, but is worse than the trivial bound of $p-1$ when $r_i$ grow  with $p$. 
A typical scenario is when we fix integers $h_i\in \Z$, $h_i\neq 0,\pm 1$, and vary the prime $p$. Weil's bound~\eqref{eq:Weil} is useful when the subgroup $\langle h_i \rangle \subseteq \Q^*$  is of rank one, that is that there is a single $h_0$ and $1\leq s_1<\ldots<s_t$ so that $h_i=h_0^{s_i}$.

 In other cases, for instance $h_1=2$ and $h_2=3$,  
we have sums $T(\va,  (r_1(p), r_2(p));p)$ where $r_1(p)$ and $r_2(p)$ depend on $p$  and  the choice of
a primitive root $g$, chosen to minimise $\max\{r_1, r_2\}$ in the representation~\eqref{eq:prime root repr}.
$$
 ( r_1, r_2) \mapsto (\lambda r_1, \lambda r_2), \qquad \lambda \in \Z_{p-1}^*,    
$$
We expect  that if $h_1$ and $h_2$ are multiplicatively independent over $\Q$ that  for almost all primes 
one should expect  
$$
\max\{r_1(p), r_2(p)\} = p^{1/2 +o(1)}, 
$$
as $p\to \infty$.
 Hence,  in this case the Weil bound ~\eqref{eq:Weil} is   trivial when applied to  $T(\va,  (r_1(p), r_2(p));p)$.

A remarkable result of Bourgain~\cite[Theorem~2]{Bourg} gives a nontrivial bound  
$$ |T(\va,\vr;p)|<p^{1-\delta} $$
when $\gcd(r_i,p-1)<p^{1-\varepsilon}$, $\gcd(r_i-r_j,p-1)<p^{1-\varepsilon}$, for some (non-explicit\footnote{An attempt in~\cite{Pop} to make the bound  of~\cite[Theorem~2]{Bourg} explicit unfortunately contains some 
unsubstantiated claims such as~\cite[Theorem~4]{Pop}, which is misquoted from~\cite{Shkr}, 
and it is not immediately clear how the effect of this slip propagates through the whole argument. }) 
$\delta >0$, which depends only on $\varepsilon$ and $t$.

Despite the large variety of   approaches to bounding the sums 
$S(\va, \vh)$  our knowledge and understanding  of these sums is still very scarce. 
Very little is known about cancellation between different sums in families parametrised by $\va$ 
or by $\vh$.  Here we mostly concentrate on the case of  {\it binomial\/} and  {\it trinomial\/} sums, that is, 
on the cases of $t=2$ and $t=3$, respectively, in which cases several results, based on various 
techniques and making various assumption on  $h_1, \ldots, h_t$ have recently been discovered, 
see~\cite{CoPin1,CoPin2,Mac,MSS,ShpVol,ShpWang} and references therein.

For integers $0 \le K < K+H \le p-1$,  we consider the following average values 
$$
U_{\va, p}(H,K):=\sum_{h_1, h_2=K+1}^{K+H}S(\va, \vh ;p)
$$
and 
$$  
V_{\va, p}(H,K):=\sum_{h_1, h_2=K+1}^{K+H} \left|S(\va, \vh ;p) \right|.
$$
The case of $K=0$ is especially interesting, so we set 
$$
U_{\va, p}(H) := U_{\va, p}(H,0) \mand 
V_{\va, p}(H) := V_{\va, p}(H,0). 
$$

Since the ultimate goal in the area is obtaining pointwise bounds we concentrate on the case 
of small values of $H$.   
Note that trivially we have
$$
|U_{\va, p}(H,K)| \le V_{\va, p}(H,K) \le H^2p.
$$

The sums  $U_{\va, p} (H,K)$ are certainly easier to treat as they admit the change of 
the order of summation after which we can apply bounds on certain double sums 
which have been studied in~\cite{FrSh, Gar1, GarKar}.

To simplify the statements of our results, we use the notation  
$$
A \lesssim B
$$
to denote that $|A| \le B p^{o(1)}$ as $p\to \infty$. 

\begin{theorem}
\label{thm:Bound UHK moderate H} 
For  any  nonzero vector $\va\in \F_p^2$, we have 
$$
U_{\va, p} (H,K) \lesssim H^{7/4} p^{9/8}.
$$
\end{theorem}

Note that Theorem~\ref{thm:Bound UHK moderate H} is nontrivial for $H>p^{1/2+\varepsilon}$. However for large $H$, namely, when $H>p^{5/6}$, we prove a stronger bound.

\begin{theorem}
\label{thm:Bound UHK} 
For  any   vector $\va\in \(\F_p^*\)^2$, we have 
$$
U_{\va, p} (H,K) \lesssim  H^{2/3} p^{2}+ H^{3/2}p^{5/4}.
$$
\end{theorem}

In particular, for $p  \lesssim  H  \lesssim p$, we have
$$
U_{\va, p} (H,K)  \lesssim H^2p\times p^{-1/4},
$$
saving $p^{1/4}$ against the trivial bound 

In the treatment of the sums $V_{\va, p} (H,K)$ our main tool is a 
result of Cochrane and Pinner~\cite{CoPin2} on binomial sums, together with 
some result on intersections of small intervals with multiplicative subgroups of $\F_p^*$,
see~\cite{BGKS,CillGar}.  In particular, we have the following bound.

\begin{theorem}
\label{thm:Bound VH} Let $n\ge 1$ be a fixed integer. 
Then for  any  vector $\va\in \(\F_p^*\)^2$, 
$$
V_{\va, p} (H)   \lesssim
\begin{cases}
H^2 p^{89/92} + H^{2}   p^{1-3/13n} &  \text{if } H \ge p^{1/n},\\
 p^{89/92+2/n} +    p^{1+23/13n}, &  \text{if } H  < p^{1/n}. 
\end{cases}
$$
\end{theorem}

We remark that our approach in the proof  of Theorem~\ref{thm:Bound VH}  is optimised to the case of small values of $H$, that is,
for large values of $n$. For example, if $p^{1/n}\lesssim H \lesssim p^{1/n}$ then 
we have 
$$
V_{\va, p} (H)   \lesssim H^2p\(p^{-3/92} + H^{-3/13} \). 
$$
For larger $H$, a different approach via bounds of character sums
can perhaps be more efficient.

\subsection{Value distribution}
We next discuss the value distribution of these sums. The main cases of interest for us is when we fix $\vh\in \Z^t$,  and the consider either   ``horizontal'' or ``vertical'' questions:  The horizontal case is when we fix $\va\in \Z^t$, say $\va=(1,\ldots,1)$ and vary $p$, and the vertical version is when for a large  prime $p$, we average   over all $\va\in \Fp^t$. 

The rank one case, when $h_i$ are all powers of a fixed $h_0$, has a different flavour, falling into an algebraic setting. For instance when  $t=2$, $h_1=2$ and $h_2=h_1^3=2^3$, we are led to consider the cubic sums 
$$
T_3(\va; p) := \sum_x \ep(a_1 x+a_2 x^3)
$$ 
for which  
Birch~\cite{Birch} conjectured a vertical   equidistribution statement, namely that 
$\{T_3(\va;p)/\sqrt{p}:~a_1\in \Fp, \ a_2\in \Fp^*\}$ is equidistributed with respect to 
the semicircle measure  $\frac 1{2 \pi}\sqrt{4-x^2}$ as $p\to \infty$,  which was proved by   Livn\'e~\cite{Livne}, see also the work of Katz~\cite{Katz1987}.  In the horizontal case, the sums $T_3((1,1);p)/\sqrt{p}$ are also conjectured to have the same distribution as $p$ varies. 
  
For higher rank, we believe the situation to be completely different. For instance, we believe that the normalized sums 
$$
\frac 1{\sqrt{p}}\sum_{x=1}^{p-1} \ep(2^x+5^x), \qquad p \to \infty,
$$ 
have a Gaussian value distribution, but this seems to be out of reach. 

We turn to the vertical version (averaging over $\va$), where we find it is more convenient to work with sums with sparse polynomials as in~\eqref{eq:sparse sums},  
bearing in mind that $r_1, \ldots, r_t$ depend on $p$. 
It is easily seen that the first moment  of $T(\va,\vr;p)$ vanishes, and the second moment equals  $p-1$. 
The first nontrivial moment is the third, which we study in the trinomial case $t=3$.   
We assume that for at least one $i=1,2,3$ we have $\gcd(r_i, p-1)=1$,  in which case we are reduced to the  sums
$$
T(\va,(1,r,s);p) :=  \sum_{x\in \F_p^*}\e_p\(a_1x^r+a_2x^s+a_3x\) ,
$$ 
with $\va=(a_1,a_2, a_3) \in \F_p^3$. 

In this case, we consider a cubic moment 
$$
M_p(r,s) := \frac 1{p^3} \sum_{\va \in \F_p^3}  \overline{T(\va,(1,r,s);p)} T(\va,(1,r,s);p)^2 ,
$$
which we believe is $o( p^{3/2})$ in a broad range of values $r$ and $s$. 

Using the orthogonality of exponential functions, we immediately evaluate that 
$$
M_p(r,s) =   Q_p(r,s),
$$
where $Q_p(r,s)$ is the number of solutions to the system of equations
$$
x^r = y^r  + z^r, \quad x^s = y^s + z^s, \quad x=y+z , \qquad x,y,z \in \F_p^*.
$$
In turn, substituting $x \to xz$, $y \to yz$, we see that 
$$
Q_p(r,s) = (p-1)N_p(r,s),
$$
where   $N_p(r,s)$ is the number of solutions to the system of equations
$$
x^r = y^r  + 1, \quad x^s = y^s + 1, \quad x=y+1 , \qquad x,y \in \F_p^*, 
$$
which equals the number of common $\F_p$-zeros of the polynomials $F_r(y)$ and $F_s(y)$, where 
\begin{equation}
\label{eq:poly F}
F_\ell(y) := (y+1)^\ell -y^\ell  - 1. 
\end{equation}

We have the following estimate on $N_p(r,s)$.

\begin{theorem}
\label{thm:Bound Nrs} There is an absolute  constant  $c$, such that for any integers $r>s> 1$
we have
$$
N_p(r,s) \le  c \frac{rs}{\log p} + 4 .
$$  
\end{theorem}   

The proof of Theorem~\ref{thm:Bound Nrs}  uses some remarkable results of Beukers~\cite{Beu}.

Note that in our application, we expect to have, typically, quantities  $N_p(r(p),s(p))$ with  
$\max\{r(p),s(p)\}\ge p^{ 1+ o(1)}$
while  Theorem~\ref{thm:Bound Nrs} is nontrivial only if 
$\max\{r,s\} = o(\log p)$. 
 Nonetheless, the following result for fixed $r,s$ may be of interest:

\begin{theorem}
\label{thm:Bound Nrs-AA} 
For any integers $r>s> 1$ for all but  $O\(rs/\log(rs)\)$ primes   we have
$$
N_p(r,s) \le 4.
$$  
\end{theorem}   

Simple counting now leads to the following result. 

\begin{cor}
\label{cor:Bound Nrs-AA} 
For any sufficiently large real $P$,  
and $R$ with 
$$R = o\(  P^{1/4} / \log P\)$$
as $P \to \infty$,
for all but $o(P/\log P)$ primes $p \le P$, for all integers 
$$
R \ge r > s > 1
$$
we have $N_p(r,s) \le 4$.
\end{cor}   

We note that by convexity (or by the H{\"o}lder inequality)
\begin{align*}
 \frac 1{p^3} \sum_{\va \in \F_p^3} |T(\va,(1,r,s);p)|^3 
&  \ge \( \frac 1{p^3} \sum_{\va \in \F_p^3} |T(\va,(1,r,s);p)|^2\)^{3/2} \\
 & \ge  (p-1)^{3/2}.
\end{align*}
Hence, when the above estimates imply that $N_p(r,s)$ is very small this also means that there are many 
cancellations in the sum defining $M_p(r,s)$. 

\section{Treatment of the sums $U_{\va, p}(H,K)$}

\subsection{Notation}
Let $\Z_N$ denote the residue ring modulo $N$. As usual, we use $\tau(r)$ and $\varphi(r)$ to denote the divisor and 
the Euler function, of an integer $k \ge 1$, respectively.  It is useful to recall the well-known estimates
\begin{equation}
\label{eqn:tauphi}
\tau(k) = k^{o(1)}\mand  \varphi(k) \gg \frac{k}{\log \log k}, 
\end{equation}
as $k\to \infty$, which we use throughout the paper. 

 We adopt the Vinogradov notation $\ll$,  that is,
$$
A\ll B~\Longleftrightarrow~A=O(B)~\Longleftrightarrow~|A|\le cB
$$
for some constant $c>0$, which  is absolute otherwise through this work.

\subsection{Some system of congruences}
\label{sec:congr}

We need the following statement.

\begin{lemma}
\label{lem:congr5var} Let $\cV\subseteq \{1,2,\ldots, p-1\}$ be a set with $V$ elements. Then the number $N$ of solutions of the system of congruences
$$
x^{v_1}\equiv h_1\bmod p \mand
x^{v_2}\equiv h_2 \bmod p 
$$
in variables 
$$
0\le x\le p-1, \quad v_1, v_2\in \cV, \quad v_1\ne v_2, \quad K+1\le h_1, h_2\le K+H,
$$
satisfies
$$
N < \frac{H^2V^2}{p}  +V^2 p^{1/2} (\log p)^2 v_{\max},
$$
where  $v_{\max}$ is the largest element of $\cV$.
\end{lemma}

\begin{proof} Using the orthogonality of exponential functions and expressing $N$ via exponential sums, we see that 
\begin{align*}
N& =
\frac{1}{p^2}\sum_{a=0}^{p-1}\sum_{b=0}^{p-1}\sum_{x=0}^{p-1}
\, \sum_{\substack{v_1. v_2\in\cV \\v_1 \ne v_2}}
\\
& \qquad \qquad \qquad   \sum_{h_1=K+1}^{K+H}\sum_{h_2=K+1}^{K+H}\ep(ax^{v_1}+bx^{v_2}-ah_1-bh_2)\\
& = \frac{1}{p^2}\sum_{a=0}^{p-1}\sum_{b=0}^{p-1}\, \sum_{\substack{v_1. v_2\in\cV \\v_1 \ne v_2}}
\(\sum_{x=0}^{p-1}\ep\(ax^{v_1}+bx^{v_2}\)\)\\
& \qquad  \qquad  \qquad  \qquad   \qquad  \qquad \(\sum_{h_1=K+1}^{K+H}\sum_{h_2=K+1}^{K+H}\ep\(-ah_1-bh_2\)\)\\
& =S_0+S_1, 
\end{align*}
where $S_0$ is a contribution from the case $a=b=0$ and $S_1$ is the sum over the remaining sum.

Obviously, 
$$
S_0 = \frac{H^2V^2}{p}.
$$

As to the  sum $S_1$,  since $v_1 \ne v_2$ in the summation in $S_1$ 
enable us  to apply the Weil bound (see, for example,~\cite[Chapter~6, Theorem~3]{Li} or~\cite[Theorem~5.38]{LN})
$$
\left|\sum_{x=0}^{p-1}\ep\(ax^{v_1}+bx^{v_2}\)\right| < v_{max} p^{1/2} .
$$
Hence,
\begin{align*}
S_1 < \frac{p^{1/2}v_{\max}}{p^2}\sum_{a=0}^{p-1}\sum_{b=0}^{p-1}\sum_{v_1\in\cV}\sum_{v_2\in\cV}
\left|\sum_{h_1=K+1}^{K+H}\sum_{h_2=K+1}^{K+H}\ep(-ah_1-bh_2)\right| =\\
\frac{V^2p^{1/2}v_{\max}}{p^2}\(\sum_{a=0}^{p-1}\left|\sum_{h=K+1}^{K+H}\ep(ah)\right|\)^2.
\end{align*}
Recalling the well-known bound
$$
\sum_{a=0}^{p-1}\left|\sum_{h=K+1}^{K+H}\ep(ah)\right| \ll  p\log p,
$$
which, for example, easily follows from~\cite[Equation~(8.6)]{IwKow}, 
we obtain that
$$
S_1 < V^2p^{1/2}(\log p)^2v_{\max}.
$$
Putting the bounds for $S_0$  and $S_1$ together, we finish the proof.
\end{proof}

\subsection{Some double exponential sums}

For  $a \in \F_p$ and a set  $\cE \subseteq \F_p$, 
we consider the following double exponential sum
$$
W_{k, a}(\cE)  = \sum_{x=1}^{p-1} \left| \sum_{e \in \cE}  \ep\(ae^x\)\right|^k. 
$$

The following statement is contained in~\cite[Corollary~2]{GarKar}.
\begin{lemma}
\label{lem:double GK} For any  $a\in \F_p^*$ and  any set  $\cE \subseteq \F_p$  
of cardinality $E$, 
we have
$$
W_{1, a}(\cE)   
 \lesssim E^{3/4} p^{9/8}.
$$
\end{lemma}

A version of estimate of Lemma~\ref{lem:double} below with an arbitrary set $\cE$ follows from~\cite[Equation~(5.2)]{Gar1} which is announced  as a consequence of the method of~\cite{Gar1}. 
Here we obtain a stronger result which takes advantage of the fact that  $\cE$ is an  interval.

\begin{lemma}
\label{lem:double} For any  $a\in \F_p^*$ and any interval
$$
\cE =[K+1, K+H] \subseteq \F_p,
$$ 
we have
$$
W_{2, a}(\cE)  \lesssim p^2H^{2/3}+p^{5/4}H^{3/2}.
$$
\end{lemma}

\begin{proof}
Obviously we can assume that $0 \not\in \cE$. We can also assume that $H>p^{3/4}$ as otherwise the claimed estimate is worse that the trivial bound $|W_{a}(\cE)|< pH^2$.
 
Let $g$ be a primitive root modulo $p$. 
Let $\cY\subseteq \Z_{p-1}$ be such that
$$
\cE = \{g^y:~y\in \cY\}
$$
In particular, $\#\cY=E$ and
$$
W_{2, a}(\cE)  = \sum_{x=1}^{p-1} \left| \sum_{y \in \cY}  \ep\(ag^{xy}\)\right|^2.
$$

For a given $d \mid p-1$, let 
\begin{align*}
\cY_d  & = \{y/d:~y\in \cY, \ d\mid y, \ \gcd(y, p-1) =d\}\\
&  =  \{1\le z\le (p-1)/d:~dz\in \cY, \ \gcd(z, (p-1)/d) =1\}.
\end{align*}  
Thus, $\cY_d$ is a subset of $\Z^{*}_{(p-1)/d}$. We have that
$$
\#\cY_d\le \min\{E, (p-1)/d\}.
$$
Since $\tau(p-1)\lesssim 1$, see~\eqref{eqn:tauphi}, we have that
$$
W_{2, a}(\cE)  \lesssim \sum_{d|p-1}\sum_{x=1}^{p-1} \left| \sum_{y \in \cY_d}  \ep\(ag^{dxy}\)\right|^2.
$$
 Therefore, again using $\tau(p-1)\lesssim 1$, we get that there exists a fixed positive integer $d \mid p-1$  such that \begin{equation}
\label{eqn:Wa(E) d is fixed}
W_{2, a}(\cE)  \lesssim \sum_{x=1}^{p-1} \left| \sum_{y \in \cY_d}  \ep\(ag^{dxy}\)\right|^2.
\end{equation}
Let $\cV$ be the set of the first $V$ positive integers relatively 
prime to $(p-1)/d$ for integer parameter $V$ to chosen later.

From the Eratosthene sieve and the bound~\eqref{eqn:tauphi}, it also follows that 
\begin{equation}
\label{eqn:vmax}
v_{\max}= \max_{v \in \cV} v\lesssim V,  
\end{equation}
see~\cite[Equation~(3.3)]{Gar1} for details.
Let
$$
\cU_d = \{1\le u\le (p-1)/d:~ \gcd(u, (p-1)/d)=1\}.
$$
That is to say, $\cU_d=\Z^*_{(p-1)/d}$. In particular, $\cY_d\subseteq \cU_d$. For any fixed $y\in \cY_d$
and $v\in\cV$, there is exactly one value $u\in\cU_d$ with
$$
uv\equiv y\pmod {(p-1/d)}.
$$
It now follows that for any fixed $y\in \cY_d$ the number of solutions of the congruence
\begin{equation}
\label{eqn:exactly V solutions}
uv\equiv y\pmod {(p-1/d)},\quad u\in \cU_d,\quad v\in \cV,
\end{equation}
is exactly equal to $V$.  Hence, if we denote by $\delta_d(y)$ the characteristic function of the set 
$\cY_d$, we obtain
\begin{align*}
\left| \sum_{y \in \cY_d}  \ep\(ag^{dxy}\)\right|^2 & =
\frac{1}{V^2}\left| \sum_{u \in \cU_d} \sum_{v\in\cV} \delta_d(uv)\ep\(ag^{dxuv}\)\right|^2\\
& \le \frac{\#\cU_d}{V^2}\sum_{u \in \cU_d} \left|\sum_{v\in\cV} \delta_d(uv)\ep\(ag^{dxuv}\)\right|^2. 
\end{align*}
Hence,
\begin{equation}
\label{eqn:getting to uv}
\sum_{x=1}^{p-1} \left| \sum_{y \in \cY_d}  \ep\(ag^{dxy}\)\right|^2\le 
\frac{\#\cU_d}{V^2}\sum_{u \in \cU_d} \sum_{x=1}^{p-1}\left|\sum_{v\in\cV} \delta_d(uv)\ep\(ag^{dxuv}\)\right|^2.
\end{equation}
If $\gcd(n, p-1)=d$ then the sequences $g^{nx}$, $x \in \Z_{p-1}$ and 
$z^d$, $z \in \Z_p^*$  run the same system of residues modulo $p$ 
(including the multiplicities).
Therefore, since $\gcd(du,p-1)=d$, from~\eqref{eqn:getting to uv} we derive that
\begin{align*}
\sum_{x=1}^{p-1} &\left| \sum_{y \in \cY_d}  \ep\(ag^{dxy}\)\right|^2\\
&\le \frac{\#\cU_d}{V^2}\sum_{u \in \cU_d} \sum_{z=1}^{p-1}\left|\sum_{v\in\cV} \delta_d(uv)\ep\(az^{dv}\)\right|^2\\
&= \frac{\#\cU_d}{V^2}\sum_{u \in \cU_d} \sum_{v_1\in\cV}\sum_{v_2\in \cV}\delta_d(uv_1)\delta_d(uv_2)
\sum_{z=1}^{p-1}\ep\(az^{dv_1}-az^{dv_2}\).
\end{align*}
The sum over $z$ is equal to $p-1$ if $v_1=v_2$ and, according to the Weil bound,
(see, for example,~\cite[Chapter~6, Theorem~3]{Li} or~\cite[Theorem~5.38]{LN})
is bounded by $\max\{v_1,v_2\}dp^{1/2}$ when $v_1\not=v_2$. Recalling that by~\eqref{eqn:vmax}
we have
$$
\max\{v_1,v_2\}\lesssim V \mand \#\cU_d =\varphi((p-1)/d)< p/d, 
$$
we see that
\begin{equation}
\label{eqn:sum with deltas}
\begin{split}
\sum_{x=1}^{p-1} &\left| \sum_{y \in \cY_d}  \ep\(ag^{dxy}\)\right|^2\\
& \lesssim
\frac{p^2}{dV^2}\sum_{u\in \cU_d} \sum_{v\in\cV}\delta_d^2(uv)+  \frac{p^{3/2}}{V} \sum_{u\in \cU_d}\, \sum_{\substack{v_1. v_2\in\cV \\v_1 \ne v_2}}\delta_d(uv_1)\delta_d(uv_2).
\end{split} 
\end{equation}
Since the congruence~\eqref{eqn:exactly V solutions} has exactly $V$ solutions, we see that
\begin{equation}
\label{eqn:sum delta  square}
\sum_{u\in \cU_d}\sum_{v\in\cV}\delta_d^2(uv)=
\sum_{u\in \cU_d}\sum_{v\in\cV}\delta_d(uv) =V \#\cY_d\le VH.
\end{equation}

To finish the proof we need to show that Lemma~\ref{lem:congr5var} implies the bound~\eqref{eqn:sum delta  from lemma}.

We now concentrate on the second sum
$$
J = \sum_{u\in \cU_d}\, \sum_{\substack{v_1. v_2\in\cV \\v_1 \ne v_2}}\delta_d(uv_1)\delta_d(uv_2).
$$
Clearly, $J$ is the number of  triples $(u,v_1,v_2)$ which satisfy
$$
uv_1\in \cY_d,\quad uv_2\in \cY_d
$$
and 
$$
1\le u\le (p-1)/d, \ \gcd(u, (p-1)/d)=1 \quad\text{and}\quad  v_1,v_2\in \cV, \ v_1\ne v_2.
$$
Recall 
\begin{align*}
\{g^{dy}&:~ y \in \cY_d\}\\
& =\{ g^{dy}\in [K+1, K+H]: \text{ for $y$ with } \gcd(y, (p-1)/d))=1\}.
\end{align*}
Therefore, every  $(u,v_1,v_2)$ which contributes to $J$ satisfies
$$
g^{duv_1}\in [K+1, K+H],\quad g^{duv_2}\in [K+1, K+H].
$$

We now notice that since  and $1\le u\le (p-1)/d$
the products $du$ are pairwise distinct modulo $p-1$. 
Hence, replacing $g^{du}$ with $x \in \F_p^*$ we see that 
$J$ does not exceed the number of triples $(x,v_1, v_2) \in \F_p^* \times \cV^2$, 
with $v_1 \ne v_2$ and such that 
$$
x^{v_1}\in [K+1, K+H], \quad x^{v_2}\in [K+1, K+H],
$$
 Now,  it only remains to recall~\eqref{eqn:vmax}
 and apply Lemma~\ref{lem:congr5var} to derive that 
\begin{equation}
\label{eqn:sum delta  from lemma}
\sum_{u\in \cU_d}\, \sum_{\substack{v_1. v_2\in\cV \\v_1 \ne v_2}}\delta_d(uv_1)\delta_d(uv_2)
\lesssim \frac{H^2V^2}{p} +V^3p^{1/2}.   
\end{equation}

Therefore, substituting~\eqref{eqn:sum delta  square} and~\eqref{eqn:sum delta  from lemma} in~\eqref{eqn:sum with deltas} we derive
\begin{equation}
\label{eqn: arriving to chose V}
\sum_{x=1}^{p-1} \left| \sum_{y \in \cY_d}  \ep\(ag^{dxy}\)\right|^2\lesssim
 \frac{p^2H}{V}+p^{1/2}H^2V + p^2V^2.
\end{equation}

To optimise the choice of $V$ in~\eqref{eqn: arriving to chose V},
 we now consider the following two cases.

First, let $H<p^{9/10}$. Then we choose $V=H^{1/3}$ and get the bound
$$
\sum_{x=1}^{p-1} \left| \sum_{y \in \cY_d}  \ep\(ag^{dxy}\)\right|^2\lesssim
 p^2H^{2/3} +p^{1/2}H^{7/3}\lesssim p^2H^{2/3}
$$
and we are done in this case.

Now, let $H > p^{9/10}$. We choose $V=p^{3/4}H^{-1/2}$. Then we obtain 
$$
\sum_{x=1}^{p-1} \left| \sum_{y \in \cY_d}  \ep\(ag^{dxy}\)\right|^2\lesssim
p^{5/4}H^{3/2}+\frac{p^{7/2}}{H}\ll p^{5/4}H^{3/2}.
$$

The result now follows from~\eqref{eqn:Wa(E) d is fixed}.
\end{proof}

\subsection{Proof of Theorems~\ref{thm:Bound UHK moderate H} and~\ref{thm:Bound UHK}}
Changing the order of summation, we write 
$$
U_{\va, p}(H,K) = \sum_{x=1}^{p-1}\sum_{h_1=K+1}^{K+H} \e_p(a_1h_1^x)
\sum_{h_2=K+1}^{K+H} \e_p(a_2h_2^x).
$$
Estimating one sum trivially as $H$  
 and Lemma~\ref{lem:double GK} 
we derive the bound of Theorem~\ref{thm:Bound UHK moderate H}. 

Next using the Cauchy inequality and Lemma~\ref{lem:double}, 
we obtain Theorem~\ref{thm:Bound UHK}

\begin{rem} We note that using  the bound of Cochrane and 
Pinner~\cite[Theorem~1.3]{CoPin2} (see Lemma~\ref{lem:binom} below), 
and the completing technique, one can sometimes get a nontrivial bound on the sums 
we estimate trivially. However,  this only works for very large values of $H$ (namely, for $H > p^{89/92+\varepsilon}$
with some fixed $\varepsilon> 0$) and this potential improvement is superseded by the bound of  Theorem~\ref{thm:Bound UHK}. 
\end{rem}

\section{Treatment of the sums   $V_{\va, p}(H)$}

\subsection{Binomial exponential sums}

For positive integers $e$ and $f$ and $a,b \in \F_p$, we consider the following 
binomial exponential sum.
$$
T_{a,b}(e,f) := \sum_{x\in \F_p} \ep\(ax^e + bx^f\).
$$
We need the following estimate of Cochrane and 
Pinner~\cite[Theorem~1.3]{CoPin2}:

\begin{lemma}
\label{lem:binom} For any integers $e$ and $f$ and $a,b \in \F_p^*$
we have
$$
|T_{a,b}(e,f)| \le \gcd(e-f, p-1) + 2.292 d^{13/46} p^{89/92},
$$
where $d = \gcd(e,f,p-1)$.
\end{lemma}

\subsection{Intersections of intervals and subgroups}

We now need some bounds on the number of integers $h$ and also of  their ratios 
$h_1/h_2$, with $h, h_1, h_2 \in  \{1,2,\ldots, H\}$ which are $d$-th power residues modulo $p$.

\begin{lemma}
\label{lem:Int-Group} 
Let $d\mid p-1$. For any fixed integer $n$ and for any positive integer  $H<p$, the number 
$I$ of integers $h\in  \{1,2,\ldots, H\}$ which are $d$-th power residues modulo $p$ satisfies
$$
I\lesssim \frac{H}{d^{1/n}} + \frac{p^{1/n}}{d^{1/n}}.
$$
\end{lemma} 

\begin{proof} Let $\cG_d$ be the subgroup of $d$-th power residues modulo $p$. Then
$$
\#\cG_d = (p-1)/d.
$$
Let  
$$
\cH_d =\{1,2,\ldots H\}  \cap \cG_d.
$$ 
(recall our convention that $\F_p$ is   represented by the elements of the set $\{0, \ldots, p-1\}$). 

For any $h_1,\ldots, h_n \in \cH_d$  we have that
$$
h_1\cdots h_n \pmod p \in \cG_d.
$$
Hence, for some fixed $g_0<p$ with $g_0\pmod p\in\cG_d$ we have that
\begin{equation}
\label{eqn: Hd power n}
I^n = (\#\cH_d)^n \le \#\cG_d \times I_n(g_0) = \frac{p-1}{d} \times I_n(g_0),
\end{equation}
where $I_n(g_0)$ is the number of solutions of the congruence
$$
h_1\cdots h_n \equiv g_0 \pmod p, \quad  h_i\in [1,H],\ i=1,...,n.
$$
It follows that
\begin{equation}
\label{eqn:h..h=g+px}
h_1\ldots h_n = g_0 + px, \quad 0\le x\le H^n/p.
\end{equation}
Since by~\eqref{eqn:tauphi} for a fixed $x$ the equation~\eqref{eqn:h..h=g+px} has at most $p^{o(1)}$ solutions, we get that
$$
I_n(g_0)\lesssim \frac{H^n}{p}+1,
$$
and the result  follows from~\eqref{eqn: Hd power n}.
\end{proof}

Next we obtain a similar result for ratios.  

\begin{lemma}
\label{lem:Rat-Group} 
Let $e\mid p-1$. For any fixed integer $n$ and for any positive integer  $H<p$, the number 
$J$ of pairs of  integers $(h, k)\in  \{1,2,\ldots, H\}^2$, for  which
the ratios $h/k$ are $e$-th power residues modulo $p$ satisfies
$$
J\lesssim H\(\frac{H}{e^{1/n}} + \frac{p^{1/n}}{e^{1/n}}\).
$$
\end{lemma} 

\begin{proof} As in the proof of Lemma~\ref{lem:Int-Group},  we notice  that $J^n$ does not exceed the number of solutions to
$$
\frac{h_1\cdots h_n}{k_1\cdots k_n}\pmod p\in \cG_e, \quad  1\le h_i, k_i\le H, \ i=1,\ldots,n, 
$$
and thus for some $ g_0 \in \cG_e$ we have 
$$
J^n \lesssim \frac{p-1}{d}\times J_n(g_0), 
$$
where $ J_n(g_0)$ is the number of solutions to the congruence
$$
 h_1\cdots h_n\equiv g_0 k_1\cdots k_n \pmod p,  \quad  1\le h_i, k_i\le H, \ i=1,\ldots,n. 
$$

We fix $(k_1,\ldots, k_n)$ in $H^n$ ways and then as  in the proof of Lemma~\ref{lem:Int-Group} 
we derive that
$$
J^n \lesssim \frac{p-1}{e}\times H^n \times \(\frac{H^n}{p}+1\)\lesssim H^n\(\frac{H^{n}}{e} + \frac{p}{e}\), 
$$
and the result follows. 
\end{proof}

\subsection{Proof of Theorem~\ref{thm:Bound VH}}
 Let $g$ be a primitive root modulo $p$. Denote
$$
\cT :=\{\ind_g h:~h\le H\},
$$
that is, 
$$
\{1,2,\ldots, H\}\ = \{g^{t}: ~t\in \cT\}\pmod p.
$$

We have that
\begin{align*}
V_{\va, p} (H)  &= \sum_{t_1\in \cT}\sum_{t_2\in \cT}\left|\sum_{x=1}^{p-1}\ep\(ag^{xt_1} + bg^{xt_2}\)\right|\\
& =
\sum_{t_1\in \cT}\sum_{t_2\in \cT}\left|\sum_{x=1}^{p-1}\ep\(ax^{t_1} + bx^{t_2}\)\right| \\
&= \sum_{d \mid p-1 } \sum_{\substack{t_1, t_2\in \cT\\ \gcd(t_1, t_2,p-1)=d}} \left|\sum_{x=1}^{p-1}\ep\(ax^{t_1} + bx^{t_2}\)\right| .
\end{align*}
Now, for two divisors $d \mid e\mid p-1$, we collect pairs $(t_1, t_2)\in \cT^2$ with the same values of 
$$
 \gcd(t_1, t_2,p-1)=d \mand  \gcd(t_1 - t_2,p-1)=e.
$$
We see from~\eqref{eqn:tauphi}  that for some  $d \mid e\mid p-1$ we have 
\begin{equation}
\label{eqn:V W}
V_{\va, p} (H)   \lesssim W
\end{equation}
where 
$$
W := \sum_{\substack{t_1\in \cT, t_2\in \cT\\ 
\gcd(t_1, t_2, p-1)=d\\
\gcd(t_1-t_2,p-1)=e}}\left|\sum_{x=1}^{p-1}\ep\(ax^{t_1} + bx^{t_2}\)\right|
$$
Combining Lemma~\ref{lem:binom}  and the trivial bound,  we obtain
\begin{equation}
\label{eqn: W BR}
W \ll B R
\end{equation}
with 
\begin{equation}
\label{eqn: Bound B}
B := \min\{B_1, B_2\}, \quad  \text{where} \ B_1:= e + d^{13/46} p^{89/92},   \  B_2 := p, 
\end{equation}
and 
$$
R :=  \sum_{\substack{t_1\in \cT, t_2\in \cT\\ 
\gcd(t_1, t_2, p-1)=d\\
\gcd(t_1-t_2,p-1)=e}} 1. 
$$
Using only the condition $\gcd(t_1, t_2, p-1)=d$ and Lemma~\ref{lem:Int-Group} we obtain 
$$
R  \lesssim  \(\frac{H}{d^{1/n}} + \frac{p^{1/n}}{d^{1/n}}\)^2
$$
while using only the condition $\gcd(t_1- t_2, p-1)=e$ and Lemma~\ref{lem:Rat-Group} we obtain 
$$
R   \lesssim H\(\frac{H}{e^{1/n}} + \frac{p^{1/n}}{e^{1/n}}\).
$$

First we consider the case   $H\ge p^{1/n+o(1)}$.
Under this  assumption  we see that we have 
\begin{equation}
\label{eqn: Bound R}
R \lesssim \min\{R_1, R_2\} \quad  \text{where} \ R_1 = H^{2} d^{-2/n}, \ R_2 =   H^{2} e^{-1/n}  . 
\end{equation}

Now, for $d \ge p^{3/26}$  we see from~\eqref{eqn: W BR}, \eqref{eqn: Bound B} and~\eqref{eqn: Bound R}
\begin{equation}
\label{eqn: Bound 1}
W  \lesssim B_2 R_1 \lesssim p \times  H^{2} d^{-2/n} \le H^{2}  p^{1-3/13n} 
\end{equation}
Similarly, for  $d \le p^{3/26}$,  we  obtain 
\begin{equation}
\label{eqn: Bound 2}
\begin{split} 
W  & \ll B_1 R\\
&   \lesssim  eR_2 + d^{13/46} p^{89/92} R_1 \\
&  \lesssim e \times  H^{2} e^{-1/n} + d^{13/46} p^{89/92} \times  H^{2} d^{-2/n} \\
& =    H^{2} e^{1-1/n} +   p^{89/92} \times  H^{2} d^{13/46-2/n} \\
 &  \lesssim H^{2} p^{1-1/n} + H^2 p^{89/92} + H^{2}   p^{1-3/13n} \\
  & \ll  H^2 p^{89/92} + H^{2}   p^{1-3/13n} .
\end{split} 
\end{equation}
Recalling~\eqref{eqn:V W}, we obtain the desired bound for  $H\ge p^{1/n+o(1)}$.

Now assume that  $H\le p^{1/n+o(1)}$. 
Under this  assumption~\eqref{eqn: Bound R}  becomes 
$$
R \lesssim \min\{\widetilde R_1, \widetilde R_2\} \quad  \text{where} \ 
\widetilde R_1 = p^{2/n} d^{-2/n}, \ 
\widetilde  R_2 =   H p^{1/n} e^{-1/n}  . 
$$
Thus, in this case for  $d \ge p^{3/26}$ we obtain 
$$
W  \lesssim B_2  \widetilde R_1   \lesssim p \times   p^{2/n} d^{-2/n}   \le p^{1+23/13n} 
$$
instead of~\eqref{eqn: Bound 1}, while instead of~\eqref{eqn: Bound 2} we obtain 
\begin{align*} 
W  & \ll B_1 R\\
&   \lesssim  e\widetilde R_2 + d^{13/46} p^{89/92} \widetilde R_1 \\
&  \lesssim e \times  H p^{1/n} e^{-1/n} + d^{13/46} p^{89/92} \times  p^{2/n} d^{-2/n}  \\
& =    H p^{1/n} e^{1-1/n} +      p^{89/92+2/n} d^{13/46-2/n} \\
 &  \lesssim H p  + p^{89/92+2/n} +    p^{1+23/13n} \\
  & \ll  p^{89/92+2/n} +    p^{1+23/13n}, 
\end{align*} 
which concludes the proof.

\section{Treatment of the number of common zeros $N_p(r,s)$}

 \subsection{Background on heights}
 There are many closely related definitions of the height on an algebraic integer. 
 It is convenient for us to use the one as in~\cite[Chapter~14]{Mass}. 
 Namely given an algebraic number $\alpha$, which is a root of an irreducible polynomial 
 $$
 f(X) = a_0 X^d + \ldots +a_{d-1} X + a_d  = a_0 \prod_{i=1}^d \(X-\alpha_i\) \in \Z[X]
 $$
 with $a_0\ge 1$ and with $\gcd(a_0, \ldots, a_d) = 1$, we define its height as 
 $$
 \HH(\alpha) := \(a_0   \prod_{i=1}^d \max\{1,  |\alpha_i|\}\)^{1/d}.
 $$
 We recall the following useful properties of the height 
 \begin{equation}
\label{eq: submult/add}
 \HH(\alpha+\beta)  \le 2  \HH(\alpha) \HH(\beta) 
 \mand
  \HH(\alpha\beta)  \le  \HH(\alpha) \HH(\beta) 
\end{equation}
and
 \begin{equation}
\label{eq: Inv}
 \HH(\alpha^{-1} )  =  \HH(\alpha) 
\end{equation}
see~\cite[Equation~(14.12)]{Mass} and~\cite[Equation~(14.24)]{Mass}, respectively.

 \subsection{Resultants and roots of some polynomials}

 We adopt a convention that  greatest common divisor of two polynomials over $\Z[X]$ is a 
 polynomial with integer co-prime coefficients and a positive leading coefficient, which makes it uniquely defined.
 
Recall  that $F_r(x) = (X+1)^r-x^r-1$, Then  (cf.~\cite{Nan,Tze1, Tze2}) 
\begin{equation}
\label{eq: FrGr}
F_r(X) = X(X+1)^a\(X^2+X+1\)^b G_r(X) , 
\end{equation}
where $\gcd\( G_r(X), X(X+1)\(X^2+X+1\)\) = 1$. Then we have
\begin{itemize}
\item if $r$ is even then  $a=b=0$;
\item if $r$ is odd then $a=1$ and 
$$
b = \begin{cases} 0, & \text{if}\ r \equiv 3 \pmod 6,\\
1, & \text{if}\ r \equiv 5 \pmod 6,\\
2, & \text{if}\ r \equiv 1 \pmod 6,
\end{cases} 
$$
\end{itemize}

We need the following result given by~\cite[Lemma~6]{Nan}. 

 \begin{lemma}
\label{lem:SimpleRoot-PolyF} 
For any integer $r  \ge 1$, the roots of $G_r(X)$, given by~\eqref{eq: FrGr}, are simple. 
\end{lemma}

 We start with recalling the following result of Beukers~\cite[Theorem~4.1]{Beu} on 
 polynomials~\eqref{eq:poly F}. 
 
 \begin{lemma}
\label{lem:GCD-PolyF} 
For any integers $r > s >1$, 
\begin{align*} 
\gcd\(F_r(X) , F_s(X)\)\in \{X, X(X+1),  X&(X+1)(X^2+X+1), \\
&X(X+1)(X^2+X+1)^2\}. 
\end{align*} 
\end{lemma}  

Next we recall a result of  Beukers~\cite[Lemma~3.5]{Beu} 
(see also~\cite[Theorem~17.1]{Mass}) 
about the height 
of roots of polynomials $F_r(X)$ 

 \begin{lemma}
\label{lem:SmallRoot-PolyF} 
For any integer $s \ge 1$, every root $\alpha$ of $F_s(X)$ is of height at most 
$\HH(\alpha) \le 216$. 
\end{lemma}  

We now recall the following result of
G\'omez-P\'erez,  Gutierrez, Ibeas and  Sevilla~\cite{GGIS}.
In fact, in the special case of square-free factors it has also been given 
in~\cite[Lemma~5.3]{KonShp} (which is also enough for our purpose).
 
 \begin{lemma} 
\label{lem:ZeroRes} 
Let $p$ be a prime number. Assume that the reductions modulo $p$ 
of two polynomials $\Psi, \Phi \in \Z[X]$ have a common factor of 
degree $m$. Then the resultant $\Res(\Psi, \Phi)$ of $\Psi$ and $\Phi$ satisfies
$$
\Res(\Psi, \Phi) \equiv 0 \pmod {p^m}.
$$
\end{lemma}

\subsection{Proof of Theorem~\ref{thm:Bound Nrs}}
Let $R_{r,s}$ be the following resultant
$$
R_{r,s} = \Res\(\Phi_r(X) , \Phi_s(X) \), 
$$
where 
$$
\Phi_r(X) := \frac{F_r(X)}{D_{r,s}(X) } \mand \Phi_s(X) :=   \frac{F_s(X)}{D_{r,s}(X) },
$$
with 
$$
D_{r,s}(X) :=  \gcd\(F_r(X) , F_s(X)\) .
$$

We see from Lemmas~\ref{lem:GCD-PolyF} and~\ref{lem:ZeroRes} that 
\begin{equation}
\label{eq: Div}
p^{N_p(r,s)-4} \mid R_{r,s}.
\end{equation}

We also see from the definition of $R_{r,s}$ that $R_{r,s}\ne 0$. Thus, to estimate $N_p(r,s)-4$ 
from~\eqref{eq: Div} we now need to obtain an upper bound on $R_{r,s}$. 

Since the leading coefficient of $F_r(X)$ is $r$, then so is of $\Phi_r(X)$. Hence 
\begin{equation}
\label{eq: Res}
|R_{r,s}| = r^{\deg \Phi_s} \prod_{\alpha:~\Phi_r(\alpha)=0}   \Phi_s(\alpha) .
\end{equation}

From Lemmas~\ref{lem:GCD-PolyF}  and~\ref{lem:SimpleRoot-PolyF} we conclude that 
\begin{equation}
\label{eq: GCD rs}
D_{r,s}(X)  \mid X(X+1)(X^2+X+1)^2. 
\end{equation}

We now see first from~\eqref{eq: Inv} and then from~\eqref{eq: GCD rs} and finally from~\eqref{eq: submult/add} that 
$$
\HH\(D_{r,s}(\alpha)^{-1}\) = \HH\(D_{r,s}(\alpha)\)  \le  \(\HH(\alpha)+2\)^{O(1)}.
$$
Hence,  by Lemma~\ref{lem:SmallRoot-PolyF}  we conclude that if $\alpha$ is a root of $F_r(\alpha)$ then 
$$
\HH\(D_{r,s}(\alpha)^{-1}\) \ll 1
$$
and thus by~\eqref{eq: submult/add} 
\begin{align*}
\HH\( \Phi_s(\alpha) \) & = \HH\(F_s(\alpha)D_{r,s}(\alpha)^{-1}\) \ll  \HH\(F_s(\alpha)\)\\
& \ll  \HH\(\alpha+1)^s\) +  \HH\(\alpha^s\) +1\\
& \le  \HH\(\alpha+1\)^s +  \HH\(\alpha\)^s +1  \ll \( \HH\(\alpha\) +2\)^s . 
\end{align*}
Using Lemma~\ref{lem:SmallRoot-PolyF} again, we derive 
$$
\HH\( \Phi_s(\alpha) \)  \le  \exp\(O(s)\), 
$$
which after substituting in~\eqref{eq: Res}  yields
\begin{equation}
\label{eq: Rrs-bound}
|R_{r,s}|  \le \exp\(O(rs)\).
\end{equation}
Using this bound together with~\eqref{eq: Div}, we conclude the proof.  

\subsection{Proof of Theorem~\ref{thm:Bound Nrs-AA}}
We see from~\eqref{eq: Div}, that if $p \nmid R_{r,s}$ then 
$N_p(r,s)\le 4$.  

As is well known, for a  nonzero integer $R$, the number $\omega(R)$ of distinct prime divisors   is bounded by 
$$
\omega(R) \ll \frac{\log |R|}{\log \log \(|R|+2\)}.
$$
Using the bound~\eqref{eq: Rrs-bound}, we conclude the proof.

\end{document}